\documentclass{amsart}
\usepackage[utf8]{inputenc}
\usepackage{enumitem}
\usepackage{amssymb}
\usepackage{comment}
\usepackage{hyperref}
\usepackage[capitalize]{cleveref}
\usepackage{thmtools}
\usepackage{mathrsfs}
\usepackage{tikz}
\usepackage{mathtools}
\usepackage{stmaryrd}
\usepackage[boxsize=0.8em, centertableaux]{ytableau}

\title{Restriction coefficients for partitions with at most three columns}
\author{Mitchell Lee}
\address{Harvard University\\
Dept. of mathematics\\
1 Oxford st.\\
Cambridge\\ MA 02138 (USA)}
\email{mitchell@math.harvard.edu}
\keywords{restriction coefficients, restriction problem, symmetric functions}
\subjclass{05E05}

\newcommand{\statement}[1]{%
  #1\enspace\ignorespaces
}

\newcommand{\K}{\mathbb{C}}
\newcommand{\Z}{\mathbb{Z}}
\newcommand{\N}{\mathbb{N}}

\newcommand{\Vect}{\text{Vect}}

\newcommand{\Par}{\mathrm{Par}}
\newcommand{\Fa}{\mathscr{F}}

\newcommand{\F}[1]{\mathscr{F}\{#1\}}

\newcommand{\Sym}{\mathfrak{S}}

\crefname{equation}{}{}
\Crefname{equation}{}{}

\DeclareMathOperator{\WC}{\mathcal{WC}}
\DeclareMathOperator{\B}{\mathcal{B}}
\DeclareMathOperator{\id}{id}

\def\multichoose#1#2{\ensuremath{\left(\kern-.3em\left(\genfrac{}{}{0pt}{}{#1}{#2}\right)\kern-.3em\right)}}

\theoremstyle{definition}
\newtheorem{defi}{Definition}[section]

\theoremstyle{plain}
\newtheorem{prop}[defi]{Proposition}
\newtheorem{theorem}[defi]{Theorem}
\newtheorem{lemma}[defi]{Lemma}
\newtheorem{coro}[defi]{Corollary}

\theoremstyle{remark}

\begin{document}

\begin{abstract}
    Let $r \geq 0$, and let $\lambda$ and $\mu$ be partitions such that $\lambda_1 \leq r + 1$. We present a combinatorial interpretation of the plethysm coefficient $\langle s_\lambda, s_\mu[s_r] \rangle$. As a consequence, we solve the \emph{restriction problem} for partitions with at most three columns. That is, for all partitions $\lambda$ with $\lambda_1 \leq 3$, we find a combinatorial interpretation for the multiplicities of the irreducible $\mathfrak{S}_n$-submodules of the Schur module $\mathbb{S}^\lambda \mathbb{C}^n$, considered as an $\mathfrak{S}_n$-module.
\end{abstract}
\maketitle
\section{Introduction}\label{sec:introduction}
For all partitions $\lambda$, let $\mathbb{S}^\lambda \colon \Vect_\K \to \Vect_\K$ denote the corresponding \emph{Schur functor}, where $\Vect_\K$ denotes the category of vector spaces over $\K$. Recall that for all positive integers $n$, we have that $\mathbb{S}^\lambda \K^n$ is a (finite-dimensional) polynomial $GL_n(\K)$-module. If $\lambda$ has at most $n$ parts, then $\mathbb{S}^\lambda \K^n$ is irreducible, and otherwise $\mathbb{S}^\lambda \K^n = 0$. All irreducible polynomial $GL_n(\K)$-modules arise in this way \cite[Section~6.1]{MR1153249}.

The composition of two Schur functors decomposes into a direct sum of Schur functors. That is, there exist nonnegative integers $p^\lambda_{\mu, \nu} \geq 0$, indexed by triples of partitions, such that for all $\mu, \nu$, there is a natural (in the vector space $V$) isomorphism
\[\mathbb{S}^\mu(\mathbb{S}^\nu V) = \bigoplus_\lambda (\mathbb{S}^\lambda V)^{\oplus p^\lambda_{\mu, \nu}}.\]

In what follows, we will use the notation $\langle s_\lambda, s_\mu[s_\nu] \rangle$ instead of $p^\lambda_{\mu, \nu}$. Here, $s_\lambda$, $s_\mu$, and $s_\nu$ denote Schur symmetric functions, $\langle \bullet, \bullet \rangle$ denotes the Hall inner product, and $f[g]$ denotes the \emph{plethysm} of the symmetric functions $f$ and $g$, which was defined by Littlewood in 1936 \cite{MR1573992}. The numbers $\langle s_\lambda, s_\mu[s_\nu] \rangle$ are called \emph{plethysm coefficients}. Plethysm coefficients have several known representation-theoretic interpretations, but finding a combinatorial interpretation of $\langle s_\lambda, s_\mu[s_\nu] \rangle$ remains a central open problem in algebraic combinatorics \cite{MR4780733,MR2765321}.

An important special case is that $\nu$ has only one row, so that $s_\nu = h_r$ is a complete homogeneous symmetric function for some $r \geq 0$. Even in this case, no combinatorial interpretation of $\langle s_\lambda, s_\mu[s_\nu] \rangle = \langle s_\lambda, s_\mu[h_r]\rangle$ is known. In this article, we present such a combinatorial interpretation in the case that $\lambda$ has at most $r + 1$ columns.

\begin{restatable}{theorem}{plethysm-hr}\label{thm:plethysm-hr}
Let $r \geq 0$, and let~$\lambda$ and~$\mu$ be partitions with $\lambda_1 \leq r + 1$.
\begin{enumerate}[label=(\alph*)]
\item If $r$ is even, $\ell(\mu) \leq r + 1$, and \[\lambda^T_i = |\mu| - \mu_{r + 2 - i}\] for $1 \leq i \leq r + 1$, then $\langle s_\lambda, s_\mu[h_r]\rangle = 1$.\label{item:plethysm-h-even}
\item If $r$ is odd, $\mu_1 \leq r + 1$, and \[\lambda^T_i = |\mu| - \mu_{r + 2 - i}^T\] for $1 \leq i \leq r + 1$, then $\langle s_\lambda, s_\mu[h_r]\rangle = 1$.\label{item:plethysm-h-odd}
\item Otherwise, $\langle s_\lambda, s_\mu[h_r]\rangle = 0$.
\end{enumerate}
\end{restatable}

An easy consequence is that if $\lambda$ has at most three columns, then we may in fact compute $\langle s_\lambda, s_\mu[h_r] \rangle$ for all $r$.

\begin{restatable}{coro}{plethysm}\label{cor:plethysm}
Let~$\lambda$ and $\mu$ be partitions with $\lambda_1 \leq 3$. For all $r \geq 4$, we have the following.

\begin{align}
\langle s_\lambda, s_\mu[h_0] \rangle &= \begin{cases} 1 & \mbox{if $\mu$ has at most one row and $\lambda = \emptyset$} \\ 0 & \mbox{otherwise} \end{cases}. \label{eq:h-zero}\\
\langle s_\lambda, s_\mu[h_1] \rangle &= \begin{cases} 1 & \mbox{if $\lambda = \mu$} \label{eq:h-one}\\ 0 & \mbox{otherwise} \end{cases}. \\
\langle s_\lambda, s_\mu[h_2] \rangle &= \begin{cases} 1 & \mbox{if $\ell(\mu) \leq 3$ and $\lambda = (\mu_1 + \mu_2, \mu_1 + \mu_3, \mu_2 + \mu_3)^T$} \\ 0 & \mbox{otherwise} \end{cases}. \label{eq:h-two} \\
\langle s_\lambda, s_\mu[h_3] \rangle &= \begin{cases} 1 & \mbox{if $\mu_1 \leq 1$ and $\lambda = (\mu_1^T, \mu_1^T, \mu_1^T)^T$} \\ 0 & \mbox{otherwise} \end{cases}. \label{eq:h-three} \\
\langle s_\lambda, s_\mu[h_r] \rangle &= \begin{cases} 1 & \mbox{if $\mu = \lambda = \emptyset$} \\ 0 & \mbox{otherwise} \end{cases}. \label{eq:h-four}
\end{align}
\end{restatable}

\Cref{cor:plethysm} allows us to solve a special case of the important \emph{restriction problem}, which we now review. Let~$\lambda$ be a partition. Because $\Sym_n$ is a subgroup of $GL_n(\K)$ via permutation matrices, we may restrict the Schur module $\mathbb{S}^\lambda \K^n$ to $\Sym_n$ and then write the result as a direct sum of irreducible $\Sym_n$-modules. In this way, we obtain, for each partition~$\mu$ with $|\mu| = n$, a positive integer \[r_{\lambda}^\mu = \dim \operatorname{Hom}_{\Sym_n}(V_\mu, \mathbb{S}^\lambda \K^n),\] where $V_\mu$ is the Specht module corresponding to the partition~$\mu$. The coefficients $r_\lambda^\mu$ are called \emph{restriction coefficients}.

In 1935, Littlewood proved that 
\begin{equation}\label{eq:littlewood}
r_{\lambda}^\mu = \langle s_\lambda, s_\mu[H] \rangle,
\end{equation} where $H = 1 + h_1 + h_2 + \cdots$ denotes the sum of all complete homogeneous symmetric functions \cite{MR1576896}. However, there remains no known combinatorial formula for $r_\lambda^\mu$. The problem of finding such a combinatorial formula is known as the \emph{restriction problem}.

The restriction problem continues to be a popular research topic and has been studied from many perspectives. Here is a small and far from comprehensive sampling of recent results. In 2021, Orellana and Zabrocki introduced the \emph{irreducible character basis} $\{\tilde{s}_\lambda\}_\lambda$ of the ring of symmetric functions and used it to provide an algorithm for computing $r_\lambda^\mu$ \cite{MR4295089}. In 2024, the author introduced an abelian group homomorphism $\Fa$ called the \emph{Frobenius transform} on the ring of symmetric functions and used it to prove several results about the vanishing of $r_\lambda^\mu$ \cite{MR4804579}. In 2024, Narayanan, Paul, Prasad, and Srivastava found a combinatorial interpretation of $r_\lambda^\mu$ in the case that~$\mu$ has one column and~$\lambda$ is either a hook shape or has at most two columns \cite{MR4804588}.

Using \cref{cor:plethysm}, we may solve the restriction problem in the case that~$\lambda$ has at most three columns. In this case, we present the following formula for $r_\lambda^\mu$. (Recall that $H = 1 + h_1 + h_2 + \cdots$.)
\begin{restatable}{theorem}{main}\label{thm:main}
Let~$\lambda$,~$\mu$ be partitions. If $\lambda_1 \leq 3$, then \[r_\lambda^\mu = \left\langle H \sum_{r, \nu} e_r s_{(\nu_1 + \nu_2 - \nu_3 - r)/2, (\nu_1 - \nu_2 + \nu_3 - r)/2, (-\nu_1 + \nu_2 + \nu_3 - r)/2} s_{\lambda / \nu^T}, s_\mu \right\rangle,\] where the sum is over all integers~$k$ and all partitions~$\nu$ such that $\nu^T \subseteq \lambda$ and $(-\nu_1 + \nu_2 + \nu_3 - r)/2$ is a nonnegative integer.
\end{restatable}

Because of the Littlewood--Richardson rule, which allows us to expand the expression \[H e_r s_{(\nu_1 + \nu_2 - \nu_3 - r)/2, (\nu_1 - \nu_2 + \nu_3 - r)/2, (-\nu_1 + \nu_2 + \nu_3 - r)/2} s_{\lambda / \nu^T} \] as a nonnegative linear combination of Schur functions, the formula presented in \cref{thm:main} is combinatorial. We will explicitly describe a combinatorial interpretation of $r_\lambda^\mu$ in the case that $\lambda$ has at most three columns in \cref{sec:interpretation}.

\section{Preliminaries}
Unless otherwise noted, the material in this section can be found in any standard reference on the theory of symmetric functions \cite{MR1464693} \cite[Chapter~I]{MR3443860} \cite[Chapter~7]{MR4621625}.

\subsection{Symmetric functions}
Let~$\Lambda$ denote the ring of symmetric functions with integer coefficients in the infinitely many variables $x_1,x_2, x_3, \ldots$.

For a sequence $\lambda = (\lambda_1, \ldots, \lambda_\ell)$ of rational numbers, define the \emph{length} $\ell(\lambda) = \ell$ and the \emph{size} $|\lambda| = \lambda_1 + \cdots + \lambda_\ell$. Let $e_\lambda = e_{\lambda_1} \cdots e_{\lambda_\ell}$ and $h_\lambda = h_{\lambda_1} \cdots h_{\lambda_\ell}$ denote the corresponding elementary and homogeneous symmetric functions, respectively, where we use the convention that $e_m = h_m = 0$ for $m \in \mathbb{Q} \setminus \N$.

If $\lambda_1, \ldots, \lambda_\ell$ are integers and $\lambda_1 \geq \cdots \geq \lambda_\ell \geq 0$, let $s_\lambda$ denote the corresponding Schur symmetric function. When a sequence $(\lambda_1, \ldots, \lambda_\ell)$ appears as a subscript in this way, we will commonly omit the parentheses; for example, $e_{1, 3, 2, 1}$ is shorthand for $e_{(1, 3, 2, 1)} = e_3 e_2 e_1^2$.

We say that $\lambda$ is a \emph{partition} if $\lambda_1, \ldots, \lambda_\ell$ are integers and $\lambda_1 \geq \cdots \geq \lambda_\ell > 0$. Let $\Par $ denote the set of all partitions. It is well known that $(e_\lambda)_{\lambda \in \Par}$, $(h_\lambda)_{\lambda \in \Par}$, and $(s_\lambda)_{\lambda \in \Par}$ are all bases for~$\Lambda$ over~$\Z$. Let $\lambda^T$ denote the transpose (i.e. conjugate) of a partition $\lambda$.

Let $\langle \bullet, \bullet \rangle \colon \Lambda \times \Lambda \to \Z$ denote the Hall inner product of symmetric functions. Let $\omega \colon \Lambda \to \Lambda$ denote the linear map given by $\omega(e_\lambda) = h_\lambda$ for all partitions~$\lambda$. It is well known that $\omega$ is a ring homomorphism and an involution and that $\omega(s_\lambda) = s_{\lambda^T}$ for all partitions~$\lambda$.

For a variable $t$, let
\begin{alignat*}{4}
H(t) &= \sum_{n \geq 0} h_n t^n &&= \prod_{i \geq 1} \frac{1}{1 - x_i t} &&= \exp\left(\sum_{k \geq 1} \frac{p_k}{k}t^k\right)&&\in \Lambda \llbracket t \rrbracket \\
E(t) &= \sum_{n \geq 0} e_n t^n &&= \prod_{i \geq 1} (1 + x_i t) &&= \exp\left(\sum_{k \geq 1} \frac{p_k}{k} (-1)^{k-1} t^k\right) &&\in \Lambda \llbracket t \rrbracket.
\end{alignat*}
It is clear that $E(-t) = (H(t))^{-1}$.

Let $\overline{\Lambda}$ denote the ring of symmetric formal power series; i.e. formal sums $\sum_{n} f_n$, where $f_n \in \Lambda$ is homogeneous of degree $n$ for all $n$. Let $H = H(1) = 1 + h_1 + h_2 + \cdots \in \overline{\Lambda}$. In the remainder of this article, we will often silently generalize definitions and results from $\Lambda$ to $\overline{\Lambda}$ using the appropriate limit. For example, even though we have explicitly stated that the Hall inner product is a function $\langle \bullet, \bullet \rangle \colon \Lambda \times \Lambda \to \Z$, we may write an expression such as \[\langle 2 + 3 h_1 - h_4, H \rangle,\] which is understood to denote the limit \[\lim_{n \to \infty} \langle 2 + 3 h_1 - h_4, 1 + h_1 + h_2 + \cdots + h_n \rangle = 4.\]
\subsection{Littlewood--Richardson coefficients}
For all partitions $\lambda, \mu, \nu$, let \[c^{\lambda}_{\mu, \nu} = \langle s_\lambda, s_\mu s_\nu\rangle.\] The numbers $c^{\lambda}_{\mu, \nu}$ are known as \emph{Littlewood--Richardson coefficients}. The well known \emph{Littlewood--Richardson rule}, which we state below, is a combinatorial interpretation of $c^{\lambda}_{\mu, \nu}$.

\begin{defi}
Let $\lambda / \mu$ be a skew shape. A \emph{Littlewood--Richardson tableau} of shape $\lambda / \mu$ is a semistandard Young tableau~$T$ of shape $\lambda / \mu$ such that the concatenation of the reversed rows of~$T$ is a lattice word; i.e., every prefix has at least as many occurrences of~$i$ as occurrences of $i + 1$ for all positive integers~$i$.
\end{defi}
\begin{theorem}[Littlewood--Richardson Rule, {\cite[Chapter~5]{MR1464693}, \cite[Chapter~9]{MR3443860}, \cite[Appendix~1]{MR4621625}}]
Let $\lambda, \mu, \nu$ be partitions. Then $c^{\lambda}_{\mu, \nu}$ is the number of Littlewood--Richardson tableaux of shape $\lambda / \mu$ and content~$\nu$. (In particular, if $\mu \not \subseteq \lambda$, then $c^{\lambda}_{\mu, \nu} = 0$.) 
\end{theorem}

Littlewood--Richardson coefficients are the structure constants for multiplication in the Schur basis. In other words, for all partitions $\mu, \nu$, we have 
\[s_\mu s_\nu = \sum_{\lambda} c^{\lambda}_{\mu, \nu} s_\lambda.\]
Moreover, for all partitions $\lambda, \mu$, the skew Schur function $s_{\lambda / \mu}$ is given by
\[s_{\lambda / \mu} = \sum_{\nu} c^{\lambda}_{\mu, \nu} s_\nu.\]

Let us define a \emph{horizontal strip} to be a skew shape $\lambda / \mu$ with no two boxes in the same column and define a \emph{vertical strip} to be a skew shape with no two boxes in the same row. One important special case of the Littlewood--Richardson rule is the following.

\begin{theorem}[Pieri's Formula, {\cite[(5.16, 5.17)]{MR3443860}, \cite[Theorem~7.15.7]{MR4621625}}]\label{thm:pieri}
Let~$\mu$ be a partition, and let $r \geq 0$. Then,
\[
    h_r s_\mu = \sum_{\substack{\lambda \\ \text{$\lambda / \mu$ is a horizontal strip} \\ |\lambda/\mu| = r}} s_\lambda
\]
and 
\[
    e_r s_\mu = \sum_{\substack{\lambda \\ \text{$\lambda / \mu$ is a vertical strip} \\ |\lambda/\mu| = r}} s_\lambda.
\]
\end{theorem}
\subsection{The coalgebra structure}\label{subsec:coalg}
In 2021, Orellana and Zabrocki used the coalgebra structure of $\Lambda$ to study the restriction problem \cite{MR4275829}. This coalgebra structure will also be useful in the proof of \cref{thm:main}. For an introduction to the material in this subsection, see the 2020 lecture notes of Grinberg and Reiner on Hopf algebras in combinatorics \cite[Chapter~2]{grinberg2020hopfalgebrascombinatorics}.

Let $\nabla \colon \Lambda \otimes \Lambda \to \Lambda$ denote the multiplication map, and let $\eta \colon \Z \to \Lambda$ denote the natural inclusion. (All tensor products in this article are over $\Z$.) Define the \emph{comultiplication} $\Delta \colon \Lambda \to \Lambda \otimes \Lambda$ and \emph{counit} $\epsilon \colon \Lambda \to \Z$ to be the $\Z$-linear maps given respectively by 
\begin{equation}\label{eq:comultiplication}\Delta(s_\lambda) = \sum_{\mu} s_\mu \otimes s_{\lambda / \mu} = \sum_{\mu, \nu} c^\lambda_{\mu, \nu} s_\mu \otimes s_\nu\end{equation}
and 
\[\epsilon(s_\lambda) = \begin{cases}
1 & \mbox{if $\lambda = \emptyset$} \\
0 & \mbox{otherwise}
\end{cases}
\] for all partitions~$\lambda$. It is well known that with the operations~$\Delta$ and~$\epsilon$, the ring~$\Lambda$ is a commutative and cocommutative graded Hopf algebra over~$\Z$. The antipode $S \colon \Lambda \to \Lambda$ is given by $S(f) = (-1)^{\deg f} \omega(f)$ for all homogeneous $f \in \Lambda$ \cite{MR506405}.

Because the Littlewood--Richardson coefficients $c^\lambda_{\mu, \nu}$ appearing in \eqref{eq:comultiplication} are also the structure constants for multiplication in the Schur basis, the comultiplication~$\Delta$ is adjoint under the Hall inner product to multiplication. That is, for all $f, g, h \in \Lambda$, we have \[\langle \Delta(f), g \otimes h\rangle = \langle f, gh\rangle,\] where the angle brackets on the left-hand side of the equation denote the bilinear form $\langle \bullet, \bullet \rangle \colon \Lambda \otimes \Lambda \times \Lambda \otimes \Lambda \to \Z$ given by \[\langle a \otimes b, c \otimes d \rangle = \langle a, c \rangle \langle b, d \rangle.\] We remark that this is one of the defining properties of a \emph{positive self-adjoint Hopf (PSH) algebra}, of which~$\Lambda$ is an example \cite[Chapter~2]{MR643482}. However, the definition of a positive self-adjoint Hopf algebra is not necessary for our proofs.

For all integers $k \geq 0$, let $\Lambda^{\otimes k}$ denote the $k$-fold tensor power of $\Lambda$; that is, \[\Lambda^{\otimes k} = \underbrace{\Lambda \otimes \cdots \otimes \Lambda}_{k}.\] Define the iterated multiplication $\nabla^{(k - 1)} \colon \Lambda^{\otimes k} \to \Lambda$ and iterated comultiplication $\Delta^{(k - 1)} \colon \Lambda \to \Lambda^{\otimes k}$ by
\begin{align*}
\nabla^{(k - 1)} &= \nabla \circ \cdots \circ (\nabla \otimes \id^{\otimes(k - 3)}) \circ (\nabla \otimes \id^{\otimes(k - 2)})\\
\Delta^{(k - 1)} &= (\Delta \otimes \id^{\otimes(k - 2)}) \circ (\Delta \otimes \id^{\otimes(k - 3)}) \circ \cdots \circ \Delta.
\end{align*}

\subsection{Plethysm}
There are two operations called ``plethysm'', both denoted using square brackets, that we will refer to in this article. The first is the plethysm of one symmetric function by another, which is classical and was mentioned in \cref{sec:introduction}. The second is the plethysm of a symmetric function by an element of a polynomial ring $\Z[t_1, \ldots, t_\ell]$, which we describe below. For more information, see the 2011 exposition of Loehr and Remmel, which describes the plethystic calculus more generally \cite{MR2765321}.
\begin{lemma}[{\cite[Example~1]{MR2765321}}]\label{lemma:plethysm}
    Let $\ell \geq 0$. There exists a unique binary operation $\bullet[\bullet] \colon \Lambda \times \Z[ t_1, \cdots, t_\ell] \to \Z[t_1, \cdots, t_\ell]$ satisfying the following properties.
    \begin{enumerate}[label=(\roman*)]
        \item For all $g \in \Z[t_1, \cdots, t_\ell]$, the function $\bullet[g] \colon \Lambda \to \Z[t_1, \cdots, t_\ell]$ is a ring homomorphism.
        \item For all $k > 0$, the function $p_k[\bullet] \colon \Z[t_1, \cdots, t_\ell] \to \Z[t_1, \cdots, t_\ell]$ is a ring homomorphism.
        \item For all~$k$ and~$i$, we have $p_k[t_i] = t_i^k$.
    \end{enumerate}
\end{lemma}
The operation from \cref{lemma:plethysm} can be expressed explicitly as a Hall inner product as follows.
\begin{lemma}[{\cite[Lemma~7.16]{MR4804579}}]\label{lemma:h-pleth}
Let $M_1, \ldots, M_N$ be monic monomials in the variables $t_1, \ldots, t_\ell$, and let $a_1, \ldots, a_N \in \Z$. For all $f \in \Lambda$, we have \begin{equation}\label{eq:h-pleth}
f\left[\sum_n a_n M_n\right] = \left\langle f, \prod_n H(M_n)^{a_n} \right\rangle.
\end{equation}
\end{lemma}
We now collect a few basic results of the plethystic calculus. Recall that $S \colon \Lambda \to \Lambda$ is the antipode (\cref{subsec:coalg}).
\begin{lemma}[{Associativity of Plethysm, \cite[Theorem~5]{MR2765321}}]\label{lemma:associativity}
    Let $f, g \in \Lambda$ and $h \in \Z[t_1, \cdots, t_\ell]$. Then, \[f[g[h]] = (f[g])[h] \in \Z[t_1, \cdots, t_\ell].\]
\end{lemma}
\begin{lemma}[{Negation Rule, \cite[Theorem~6]{MR2765321}}]\label{lemma:negation}
    Let $f, g \in \Lambda$. Then, \[f[-g] = S(f)[g].\]
\end{lemma}
\begin{lemma}[{Monomial Substitution Rule, \cite[Theorem~7]{MR2765321}}]\label{lemma:monomialsubstitution}
    Let $f \in \Lambda$ and let $M_1, \ldots, M_N \in \Z[t_1, \ldots, t_\ell]$ be monic monomials. Then,
    \[f\left[\sum_n M_n\right] = f(M_1, \ldots, M_N).\]
\end{lemma}
\subsection{The plethystic addition formula}
We will make heavy use of the following formula.
\begin{prop}[Plethystic Addition Formula, {\cite[(8.8)]{MR3443860} \cite[Section~3.2]{MR2765321}}]\label{prop:plethystic-addition}
    Let~$\lambda$ be a partition and let $f, g \in \Lambda$. Then,
    \[s_{\lambda}[f + g] = \sum_{\mu} s_{\mu}[f] s_{\lambda / \mu}[g].\]
\end{prop}
For our purposes, it will be useful to rephrase \cref{prop:plethystic-addition} in terms of multiplication and comultiplication. To this end, for all $g \in \Lambda$, let $\bullet[g] \colon \Lambda \to \Lambda$ denote plethysm by $g$, which is a ring homomorphism. Explicitly, \[(\bullet[g])(f) = f[g]\] for all $f \in \Lambda$. \cref{prop:plethystic-addition} can then be written in the following form.

\begin{prop}\label{prop:plethystic-addition-2}
For all $f, g \in \Lambda$, we have \[\bullet[f + g] = \nabla \circ (\bullet[f] \otimes \bullet[g]) \circ \Delta.\]
\end{prop}
\begin{proof}
It suffices to show that \[(\bullet[f + g])(s_\lambda) = (\nabla \circ (\bullet[f] \otimes \bullet[g]) \circ \Delta)(s_\lambda)\] for all partitions $\lambda$. After unfolding the definitions, this is exactly \cref{prop:plethystic-addition}.
\end{proof}
By applying \cref{prop:plethystic-addition-2} repeatedly, we obtain the following generalization.
\begin{prop}\label{prop:plethystic-addition-3}
Let $k \geq 0$, and let $g_1, \ldots, g_k \in \Lambda$. Then, \[\bullet[g_1 + \cdots + g_k] = \nabla^{(k - 1)} \circ (\bullet[g_1] \otimes \cdots \otimes \bullet[g_k]) \circ \Delta^{(k - 1)}.\]
\end{prop}
\subsection{The Jacobi--Trudi identities}
A key ingredient in the proof of \cref{thm:main} is the well known Jacobi--Trudi identities, which we now recall.
\begin{theorem}[Jacobi--Trudi Identities, {\cite[(5.5)]{MR3443860}, \cite[Section~7.16]{MR4621625}}]\label{thm:jacobi-trudi}
Let $\ell \geq 0$, and let~$\lambda \in \N^\ell$ with $\lambda_1 \geq \cdots \geq \lambda_\ell$. Then,
\begin{equation}\label{eq:first-jacobi-trudi}s_{\lambda} = \det(h_{\lambda_i - i + j})_{i, j = 1}^{\ell}\end{equation} and
\begin{equation}\label{eq:second-jacobi-trudi}s_{\lambda^T} = \det(e_{\lambda_i - i + j})_{i, j = 1}^{\ell}.\end{equation}
\end{theorem}

\section{Proof of \texorpdfstring{\cref{cor:plethysm,thm:main}}{Theorems~\ref{cor:plethysm}~and~\ref{thm:main}}}
\subsection{Plethysm adjoints}
For all $g \in \overline{\Lambda}$, let $\bullet[g^\perp] \colon \Lambda \to \overline{\Lambda}$ denote the adjoint to plethysm by $g$ under the Hall inner product. In other words,
\begin{equation}\label{eq:f-g-perp}f[g^\perp] = \sum_{\mu} \langle f, s_\mu[g]\rangle s_\mu\end{equation} for all $f \in \Lambda$. The notation $\bullet[g^\perp]$ was introduced by Ryba in 2021 \cite{MR4658448}.

Because $\bullet[g]$ is an algebra homomorphism, $\bullet[g^\perp]$ is a coalgebra homomorphism. Moreover, by taking the adjoint of both sides of \cref{prop:plethystic-addition-3}, we obtain the following.
\begin{prop}[{\cite[Proposition~2.5]{MR4658448}}]\label{prop:plethystic-addition-4}
Let $k \geq 0$, and let $g_1, \ldots, g_k \in \Lambda$. Then, \[\bullet[(g_1 + \cdots + g_k)^\perp] = \Delta^{(k - 1)} \circ (\bullet[g_1^\perp] \otimes \cdots \otimes \bullet[g_k^\perp]) \circ \nabla^{(k - 1)}.\]
\end{prop}

Of special note is the plethysm adjoint \[\Fa = \bullet[H^\perp]\colon \Lambda \to \overline{\Lambda},\] which the author has referred to as the \emph{Frobenius transform} in 2024 \cite[Remark~3.2]{MR4804579}. By \eqref{eq:littlewood}, the Frobenius transform is given on the Schur basis by
\[s_\lambda[H^\perp] = \sum_\mu r^\mu_\lambda s_\mu.\]
Thus, \cref{thm:main} is equivalent to the statement that for all partitions $\lambda$ with $\lambda_1 \leq 3$, we have
\begin{equation}\label{eq:main}
    s_\lambda[H^\perp] = H \sum_{r, \nu} e_r s_{(\nu_1 + \nu_2 - \nu_3 - r)/2, (\nu_1 - \nu_2 + \nu_3 - r)/2, (-\nu_1 + \nu_2 + \nu_3 - r)/2} s_{\lambda / \nu^T},
\end{equation}
where the sum is over all integers~$r$ and all partitions~$\nu$ such that $\nu^T \subseteq \lambda$ and $(-\nu_1 + \nu_2 + \nu_3 - r)/2$ is a nonnegative integer. In what follows, our goal will be to prove \eqref{eq:main}.

Let us recall the following formulas for $\F{h_\lambda} = h_\lambda[H^\perp]$ and $\F{e_\lambda} = e_\lambda[H^\perp]$, where~$\lambda$ is a partition.

\begin{theorem}[{\cite[Equation (6)]{MR4245137}, \cite[Theorem~1.4(a,b)]{MR4804579}}]\label{thm:f-he}
Let $\ell \geq 0$, and let $\lambda \in \N^\ell$.
\begin{enumerate}[label=(\alph*)]
\item \label{item:f-h} We have \[h_\lambda[H^\perp] = \sum_{M} \prod_{j \in \N^\ell} h_{M(j)},\] where the sum is over all functions $M\colon \N^\ell \to \N$ such that \[\sum_{j \in \N^\ell} M(j) j = \lambda.\]
\item \label{item:f-e} We have \[e_\lambda[H^\perp] = \sum_{M} \prod_{j \in \{0, 1\}^\ell}\begin{cases}
h_{M(j)} & \mbox{if $j_1 + \cdots + j_\ell$ is even} \\
e_{M(j)} & \mbox{if $j_1 + \cdots + j_\ell$ is odd}
\end{cases},\] where the sum is over all functions $M \colon \{0, 1\}^\ell \to \N$ such that \[\sum_{j \in \{0, 1\}^\ell} M(j) j = \lambda.\]
\end{enumerate}
\end{theorem}

For $r, \ell \geq 0$, let \[\WC(\ell, r) = \{(j_1, \ldots, j_\ell) \in \N^\ell \, | \, j_1 + \cdots + j_\ell = r\}\] denote the set of all weak compositions of $r$ into $\ell$ parts, and let \[\B(\ell, r) = \{(j_1, \ldots, j_\ell) \in \{0, 1\}^\ell \, | \, j_1 + \cdots + j_\ell = r\}\] denote the set of all binary vectors of weight $r$ and length $\ell$. To prove \eqref{eq:main}, we will use the following new formula for the plethysm adjoints $h_\lambda[h_r^\perp]$ and $e_\lambda[h_r^\perp]$. We will only use part \labelcref{item:e-h-perp} of \cref{thm:he-h-perp} in what follows, but we include both parts because they may be of some independent interest.

\begin{theorem}\label{thm:he-h-perp}
Let $\ell, r \geq 0$, and let $\lambda \in \N^\ell$.
\begin{enumerate}[label=(\alph*)]
\item \label{item:h-h-perp} We have \[h_\lambda[h_r^\perp] = \sum_{M} \prod_{j \in \WC(\ell, r)} h_{M(j)},\] where the sum is over all functions $M\colon \WC(\ell, r) \to \N$ such that \[\sum_{j \in \WC(\ell, r)} M(j) j = \lambda.\]
\item \label{item:e-h-perp} We have \[e_\lambda[h_r^\perp] = \sum_{M} \prod_{j \in \B(\ell, r)}\begin{cases}
h_{M(j)} & \mbox{if $r$ is even} \\
e_{M(j)} & \mbox{if $r$ is odd}
\end{cases},\] where the sum is over all functions $M \colon \B(\ell, r) \to \N$ such that \[\sum_{j \in \B(\ell, r)} M(j) j = \lambda.\]
\end{enumerate}
\end{theorem}

\begin{proof}
\statement{(a)} Let $t_1, \ldots, t_\ell$ be variables. Observe that for all $f \in \Lambda$, we have 
\begin{align}
\langle f, (H(t_1) \cdots H(t_\ell))[h_r^\perp]\rangle &= \langle f[h_r], H(t_1) \cdots H(t_\ell) \rangle\label{eq:prod-wc-1}\\
&= (f[h_r])[t_1 + \cdots + t_\ell]\label{eq:prod-wc-2}\\
&= f[h_r[t_1 + \cdots + t_\ell]]\label{eq:prod-wc-3}\\
&= f\left[\sum_{j \in \WC(\ell, r)}t_1^{j_1} \cdots t_\ell^{j_\ell}\right]\label{eq:prod-wc-4}\\
&= \left\langle f, \prod_{j \in \WC(\ell, r)} H(t_1^{j_1} \cdots t_\ell^{j_\ell}) \right\rangle\label{eq:prod-wc-5}.
\end{align}
For \eqref{eq:prod-wc-1}, we used the definition of plethysm adjoints. For \eqref{eq:prod-wc-2}, we used \cref{lemma:h-pleth}. For \eqref{eq:prod-wc-3}, we used the associativity of plethysm (\cref{lemma:associativity}). For \eqref{eq:prod-wc-4}, we used the monomial substitution rule (\cref{lemma:monomialsubstitution}). For \eqref{eq:prod-wc-5}, we used \cref{lemma:h-pleth} again.

It follows that
\begin{equation*}(H(t_1) \cdots H(t_\ell))[h_r^\perp] = \prod_{j \in \WC(\ell, r)} H(t_1^{j_1} \cdots t_\ell^{j_\ell}).\end{equation*}
Taking the coefficient of $t_1^{\lambda_1} \cdots t_\ell^{\lambda_\ell}$ on both sides yields the desired result. 

\statement{(b)} Let $t_1, \ldots, t_\ell$ be variables. Observe that for all $f \in \Lambda$, we have
\begin{align}
\langle f, (H(t_1)^{-1} \cdots H(t_\ell)^{-1})[h_r^\perp]\rangle  &= \langle f[h_r], H(t_1)^{-1} \cdots H(t_\ell)^{-1} \rangle\label{eq:prod-b-1}\\
&= (f[h_r])[-(t_1 + \cdots + t_\ell)]\label{eq:prod-b-2}\\
&= f[h_r[-(t_1 + \cdots + t_\ell)]]\label{eq:prod-b-3}\\
&= f\left[(-1)^r e_r[t_1 + \cdots + t_\ell]\right]\label{eq:prod-b-4}\\
&= f\left[(-1)^r \sum_{j \in \B(\ell, r)} t_1^{j_1} \cdots t_\ell^{j_\ell}\right] \label{eq:prod-b-5}\\
&= \left\langle f, \prod_{j \in \B(\ell, r)} H(t_1^{j_1} \cdots t_\ell^{j_\ell})^{(-1)^r} \right\rangle\label{eq:prod-b-6}.
\end{align}
For \eqref{eq:prod-b-1}, we used the definition of plethysm adjoints. For \eqref{eq:prod-b-2}, we used \cref{lemma:h-pleth}. For \eqref{eq:prod-b-3}, we used the associativity of plethysm (\cref{lemma:associativity}). For \eqref{eq:prod-b-4}, we used the negation rule (\cref{lemma:negation}). For \eqref{eq:prod-b-5}, we used the monomial substitution rule (\cref{lemma:monomialsubstitution}). For \eqref{eq:prod-b-6}, we used \cref{lemma:h-pleth} again.

It follows that
\begin{equation}\label{eq:prod-b}
(H(t_1)^{-1} \cdots H(t_\ell)^{-1})[h_r^\perp] = \prod_{j \in \B(\ell, r)} H(t_1^{j_1} \cdots t_\ell^{j_\ell})^{(-1)^r}.
\end{equation}
Using the equation $E(-t) = (H(t))^{-1}$, the left-hand side of \eqref{eq:prod-b} simplifies to $(E(-t_1) \cdots E(-t_\ell))[h_r^\perp]$ and the right-hand side simplifies to 
\[\begin{cases}
\prod_{j \in \B(\ell, r)} H(t_1^{j_1} \cdots t_\ell^{j_\ell}) & \mbox{if $r$ is even} \\
\prod_{j \in \B(\ell, r)} E(-t_1^{j_1} \cdots t_\ell^{j_\ell}) & \mbox{if $r$ is odd}
\end{cases}.\] Taking the coefficient of $(-1)^{|\lambda|}t_1^{\lambda_1} \cdots t_\ell^{\lambda_\ell}$ on both sides yields the desired result. 
\end{proof}

\subsection{Specializing to partitions with at most \texorpdfstring{$r + 1$}{r + 1} columns}
In this subsection, we will prove \cref{thm:plethysm-hr} and \cref{cor:plethysm}. First, consider the following special case of \cref{thm:he-h-perp}\labelcref{item:e-h-perp}.
\begin{lemma}\label{lemma:e-lambda-hr}
Let $r \geq 1$, and let $\lambda \in \Z^{r + 1}$. Then,
\[e_\lambda[h_r^\perp] = \begin{cases}
\prod_{i = 1}^{r + 1} h_{|\lambda|/r - \lambda_i} & \mbox{if $r$ is even} \\
\prod_{i = 1}^{r + 1} e_{|\lambda|/r - \lambda_i} & \mbox{if $r$ is odd}\end{cases}.\]
Recall that we use the convention that $h_m = e_m = 0$ for all $m \in \mathbb{Q} \setminus \N$.
\end{lemma}
\begin{proof}
Apply \cref{thm:he-h-perp}\labelcref{item:e-h-perp} with $\ell = r + 1$. Note that $\B(r + 1, r) = \{j_1, \ldots, j_{r + 1}\}$, where $j_i \in \N^{r + 1}$ denotes the vector whose $i$th component is $0$ and whose other components are $1$. The condition \[\sum_{j \in \B(r + 1, r)} M(j) j = \lambda\] appearing in \cref{thm:he-h-perp}\labelcref{item:e-h-perp} is equivalent to the statement that
\[(M(j_1) + \cdots + M(j_{r+1})) - M(j_i) = \lambda_i\] for all $i$.
This system of equations only has one solution; namely, \[M(j_i) = |\lambda|/r - \lambda_i.\] The lemma follows.
\end{proof}
\begin{proof}[Proof of \cref{thm:plethysm-hr}]
Our strategy will be to compute $s_\lambda[h_r^\perp]$. By the second Jacobi--Trudi identity \eqref{eq:second-jacobi-trudi}, we have
\begin{align*}
s_\lambda &= \det(e_{\lambda^T_i - i + j})_{i, j = 1}^{r + 1} \\
&= \sum_{\pi \in \Sym_{r + 1}} (-1)^{\pi} e_{\lambda^T - \Delta + \pi},
\end{align*}
where $\Delta = (1, 2, 3, \ldots, r + 1)$ is the reverse staircase shape, and we consider $\pi$ as an element of $\Z^{r + 1}$ via one-line notation.

Therefore, by \cref{lemma:e-lambda-hr}, we have
\begin{align*}
s_\lambda[h_r^\perp] &= \sum_{\pi \in \Sym_{r + 1}} (-1)^{\pi} e_{\lambda^T - \Delta + \pi} [h_r^\perp]\\
&= \sum_{\pi \in \Sym_{r + 1}} (-1)^{\pi} \begin{cases}
\prod_{i = 1}^{r + 1} h_{|\lambda - \Delta + \pi|/r - (\lambda_i - i + \pi(i))} & \mbox{if $r$ is even} \\
\prod_{i = 1}^{r + 1} e_{|\lambda - \Delta + \pi|/r - (\lambda_i - i + \pi(i))} & \mbox{if $r$ is odd}
\end{cases} \\
&= \sum_{\pi \in \Sym_{r + 1}} (-1)^{\pi} \begin{cases}
\prod_{i = 1}^{r + 1} h_{\left(|\lambda|/r - \lambda_i\right) + i - \pi(i)} & \mbox{if $r$ is even} \\
\prod_{i = 1}^{r + 1} e_{\left(|\lambda|/r - \lambda_i\right) + i - \pi(i)} & \mbox{if $r$ is odd}
\end{cases}.
\end{align*}
Let us now simultaneously perform the changes of variables $i \mapsto r + 2 - i$ and $\pi \mapsto \pi^{rc}$ in the above product and sum respectively, where $\pi^{rc}$ is the reverse complement of $\pi$, given by $\pi^{rc}(i) = (r + 2) - \pi(r + 2 - i)$ for $1 \leq i \leq r + 1$. This yields
\begin{align}
s_\lambda[h_r^\perp] &= \sum_{\pi \in \Sym_{r + 1}} (-1)^{\pi} \begin{cases}
\prod_{i = 1}^{r + 1} h_{\left(|\lambda|/r - \lambda_{r + 2 - i}\right) - i + \pi(i)} & \mbox{if $r$ is even} \\
\prod_{i = 1}^{r + 1} e_{\left(|\lambda|/r - \lambda_{r + 2 - i}\right) - i + \pi(i)} & \mbox{if $r$ is odd}
\end{cases} \nonumber\\
&= \begin{cases}
\det\left(h_{\left(|\lambda|/r - \lambda_{r + 2 - i}\right) - i + j}\right)_{i, j = 1}^{r + 1} & \mbox{if $r$ is even} \\
\det\left(e_{\left(|\lambda|/r - \lambda_{r + 2 - i}\right) - i + j}\right)_{i, j = 1}^{r + 1} & \mbox{if $r$ is odd}
\end{cases}\label{eq:det-complement}
\end{align}

If $|\lambda| / r - \lambda_1 \not \in \N$, then the right-hand side of \eqref{eq:det-complement} is the determinant of a matrix whose last row is $0$, so $\langle s_\lambda, s_\mu[h_r] \rangle = \langle s_\lambda[h_r^\perp], s_\mu \rangle = 0$. It is clear that if $|\lambda| / r - \lambda_1 \not \in \N$, then the hypotheses of \cref{thm:plethysm-hr}\labelcref{item:plethysm-h-even} and the hypotheses of \cref{thm:plethysm-hr}\labelcref{item:plethysm-h-odd} are both impossible, so the theorem is proved.

Hence, we may assume that $|\lambda| / r - \lambda_1 \in \N$, which implies that \[\left(|\lambda|/r - \lambda_{r + 1}, \ldots, |\lambda|/r - \lambda_{1}\right)\] is a partition with zeroes appended to the end. In this case, we may simplify the right-hand side of \eqref{eq:det-complement} using the Jacobi--Trudi identities (\cref{thm:jacobi-trudi}) to obtain
\[s_{\lambda}[h_r^\perp] = \begin{cases}
s_{\left(|\lambda|/r - \lambda_{r + 1}, \ldots, |\lambda|/r - \lambda_{1}\right)} & \mbox{if $r$ is even} \\
s_{\left(|\lambda|/r - \lambda_{r + 1}, \ldots, |\lambda|/r - \lambda_{1}\right)^T} & \mbox{if $r$ is odd}
\end{cases}.\]
We conclude that if $r$ is even, then \[\langle s_\lambda, s_\mu[h_r] \rangle = \begin{cases}
1 & \mbox{if $\mu = \left(|\lambda|/r - \lambda_{r + 1}, \ldots, |\lambda|/r - \lambda_{1}\right)$} \\
0 & \mbox{otherwise}
\end{cases}\] and that if $r$ is odd, then \[\langle s_\lambda, s_\mu[h_r] \rangle = \begin{cases}
1 & \mbox{if $\mu = \left(|\lambda|/r - \lambda_{r + 1}, \ldots, |\lambda|/r - \lambda_{1}\right)^T$} \\
0 & \mbox{otherwise}
\end{cases}.\] This is easily seen to be equivalent to the theorem. 
\end{proof}
\begin{proof}[Proof of \cref{cor:plethysm}]
The first two equations follow from the basic properties of plethysm, and the other three follow from \cref{thm:plethysm-hr}.
\end{proof}
In what follows, we will primarily use the following form of \cref{cor:plethysm}, which is the result of using \eqref{eq:f-g-perp}.
\begin{coro}\label{thm:s-lambda-h-perp}
Let $\lambda$ be a partition with $\lambda_1 \leq 3$. For all $r \geq 4$, we have the following.
\begin{align}
\label{eq:s-h0-perp} s_\lambda[h_0^\perp] &= \begin{cases}
H & \mbox{if $\lambda = \emptyset$} \\
0 & \mbox{otherwise}
\end{cases}. \\
\label{eq:s-h1-perp} s_\lambda[h_1^\perp] &= s_\lambda. \\
\label{eq:s-h2-perp} s_\lambda[h_2^\perp] &= s_{(\lambda^T_1 + \lambda^T_2 - \lambda^T_3)/2, (\lambda^T_1 - \lambda^T_2 + \lambda^T_3)/2, (-\lambda^T_1 + \lambda^T_2 + \lambda^T_3)/2}. \\
\label{eq:s-h3-perp} s_\lambda[h_3^\perp] &= \begin{cases}
e_{\lambda^T_1} & \mbox{if $\lambda^T_1 = \lambda^T_2 = \lambda^T_3$} \\
0 & \mbox{otherwise}
\end{cases}. \\
\label{eq:s-hr-perp} s_\lambda[h_r^\perp] &= \begin{cases}
1 & \mbox{if $\lambda = \emptyset$} \\
0 & \mbox{otherwise}
\end{cases}.
\end{align}
In \eqref{eq:s-h2-perp}, we use the convention that \[s_{(\lambda^T_1 + \lambda^T_2 - \lambda^T_3)/2, (\lambda^T_1 - \lambda^T_2 + \lambda^T_3)/2, (-\lambda^T_1 + \lambda^T_2 + \lambda^T_3)/2} = 0\] if $(-\lambda^T_1 + \lambda^T_2 + \lambda^T_3)/2 \not \in \N$.
\end{coro}
\subsection{Performing the comultiplication}
In this subsection, we will complete the proof of \cref{thm:main}. Let \[D = \operatorname{span} \{s_\lambda \, | \, \lambda_1 \leq 3\} \subseteq \Lambda.\] Because the Littlewood-Richardson coefficient $c^{\lambda}_{\mu, \nu}$ is zero if $\mu_1 > \lambda_1$ or $\nu_1 > \lambda_1$, we have that $D$ is a \emph{subcoalgebra} of $\Lambda$, meaning that $\Delta(D) \subseteq D \otimes D$. Thus,
\begin{equation}\label{eq:delta-k-d}
\Delta^{(k)}(D) \subseteq D^{\otimes (k + 1)}
\end{equation} for all $k$.

By \eqref{eq:s-hr-perp}, we have $f[h_r^\perp] = \eta(\epsilon(f))$ for all $f \in D$ and all $r \geq 4$. (Recall that $\eta$ and $\epsilon$ are the unit and counit of $\Lambda$, respectively; see \cref{subsec:coalg}.) Therefore, by \cref{prop:plethystic-addition-3}, for all $f \in D$, we have \begin{align}
\nonumber f[H^\perp] &= \lim_{k \to \infty} f[(1 + h_1 + \cdots + h_k)^\perp]\\
\nonumber &= \lim_{k \to \infty} (\nabla^{(k)} \circ (\bullet[1^\perp] \otimes \bullet[h_1^\perp] \otimes \cdots \otimes \bullet[h_k^\perp]) \circ \Delta^{(k)})(f) \\
\nonumber &= \lim_{k \to \infty} (\nabla^{(k)} \circ (\bullet[1^\perp] \otimes \bullet[h_1^\perp] \otimes \bullet[h_2^\perp] \otimes \bullet[h_3^\perp] \otimes (\eta \circ \epsilon)^{\otimes k - 3}) \circ \Delta^{(k)})(f) \\
\nonumber &= (\nabla^{(3)} \circ (\bullet[1^\perp] \otimes \bullet[h_1^\perp] \otimes \bullet[h_2^\perp] \otimes \bullet[h_3^\perp]) \circ \Delta^{(3)})(f) \\
\label{eq:frobenius-eq-mul-comul} &= (\nabla^{(3)} \circ (\bullet[1^\perp] \otimes \bullet[h_3^\perp] \otimes \bullet[h_2^\perp] \otimes \bullet[h_1^\perp]) \circ \Delta^{(3)})(f).
\end{align}
Let us now apply \eqref{eq:frobenius-eq-mul-comul} to the Schur function $f = s_\lambda$, where~$\lambda$ is a partition with $\lambda_1 \leq 3$. By the definition of comultiplication \eqref{eq:comultiplication}, we have
\begin{equation}\label{eq:comul-3-s}\Delta^{(3)}(s_\lambda) = \sum_{\lambda^{(1)}, \lambda^{(2)}, \lambda^{(3)}} s_{\lambda^{(1)}} \otimes s_{\lambda^{(2)} / \lambda^{(1)}} \otimes s_{\lambda^{(3)} / \lambda^{(2)}} \otimes s_{\lambda / \lambda^{(3)}},\end{equation} where the sum is over all partitions $\lambda^{(1)}, \lambda^{(2)}, \lambda^{(3)}$. Therefore,
\begin{align}
    \label{eq:s-sum-3}s_\lambda[H^\perp] &= \sum_{\lambda^{(1)}, \lambda^{(2)}, \lambda^{(3)}} s_{\lambda^{(1)}}[1^\perp] \cdot s_{\lambda^{(2)} / \lambda^{(1)}}[h_3^\perp] \cdot s_{\lambda^{(3)} / \lambda^{(2)}}[h_2^\perp] \cdot s_{\lambda / \lambda^{(3)}}[h_1^\perp] \\
    \label{eq:s-sum-2} &= H \sum_{\lambda^{(2)}, \lambda^{(3)}} s_{\lambda^{(2)}}[h_3^\perp] \cdot s_{\lambda^{(3)} / \lambda^{(2)}}[h_2^\perp] \cdot s_{\lambda / \lambda^{(3)}}[h_1^\perp] \\
    \label{eq:s-sum-1} &= H \sum_{r, \lambda^{(3)}} e_r \cdot s_{\lambda^{(3)} / (r, r, r)^T}[h_2^\perp] \cdot s_{\lambda / \lambda^{(3)}}[h_1^\perp] \\
    \label{eq:s-sum-1'} &= H \sum_{r, \lambda^{(3)}} e_r \cdot s_{\lambda^{(3)} / (r, r, r)^T}[h_2^\perp] \cdot s_{\lambda / \lambda^{(3)}}
\end{align}
where the sum is over all integers $r \geq 0$. For \eqref{eq:s-sum-3}, we used \cref{eq:frobenius-eq-mul-comul,eq:comul-3-s}. For \eqref{eq:s-sum-2}, we used \cref{eq:s-h0-perp}. For \eqref{eq:s-sum-1}, we used \cref{eq:s-h3-perp}. For \eqref{eq:s-sum-1'}, we used \cref{eq:s-h1-perp}.

Finally, we reindex the sum in \cref{eq:s-sum-1'} by taking $\nu = (\lambda^{(3)})^T$, and we apply \cref{eq:s-h2-perp}. This yields
\begin{align*}
s_\lambda[H^\perp] &= H \sum_{r, \nu} e_r \cdot s_{(\nu_1 - r, \nu_2 - r, \nu_3 - r)^T}[h_2^\perp] \cdot s_{\lambda / \nu^T} \\
&= H \sum_{r, \nu} e_r s_{(\nu_1 + \nu_2 - \nu_3 - r)/2, (\nu_1 - \nu_2 + \nu_3 - r)/2, (-\nu_1 + \nu_2 + \nu_3 - r)/2} s_{\lambda / \nu^T},
\end{align*}
proving \eqref{eq:main} and thus \cref{thm:main}.

\section{The combinatorial interpretation}\label{sec:interpretation}
We may now derive a combinatorial interpretation of the restriction coefficient $r_\lambda^\mu$, where $\lambda, \mu$ are partitions with $\lambda_1 \leq 3$.
\begin{coro}
Let $\lambda, \mu$ be partitions with $\lambda_1 \leq 3$. Then, $r_\lambda^\mu$ is the number of tuples $(r, \nu, \lambda^{(1)}, \lambda^{(2)}, \lambda^{(3)}, T^{(1)}, T^{(2)})$, where
\begin{itemize}
\item $r \geq 0$ is an integer;
\item $\nu, \lambda^{(1)}, \lambda^{(2)}, \lambda^{(3)}$ are partitions;
\item $(-\nu_1 + \nu_2 + \nu_3 - r) / 2$ is a nonnegative integer;
\item $T^{(1)}$ is a Littlewood--Richardson tableau of shape $\lambda / \nu^T$ and content $\lambda^{(1)}$;
\item $T^{(2)}$ is a Littlewood--Richardson tableau of shape $\lambda^{(2)} / \lambda^{(1)}$ and content $((\nu_1 + \nu_2 - \nu_3 - r)/2, (\nu_1 - \nu_2 + \nu_3 - r)/2, (-\nu_1 + \nu_2 + \nu_3 - r)/2)$;
\item $\lambda^{(3)} / \lambda^{(2)}$ is a vertical strip with $r$ boxes;
\item $\mu / \lambda^{(3)}$ is a horizontal strip.
\end{itemize}
\end{coro}
\begin{proof}
This is the result of starting with \cref{thm:main}, expanding each skew Schur function and product of Schur functions using Littlewood--Richardson coefficients, and expanding each product of an elementary or complete homogeneous symmetric function with a Schur function using Pieri's formula (\cref{thm:pieri}).
\end{proof}
\bibliographystyle{plain}
\bibliography{main}
\end{document}